\documentclass[12pt, reqno]{amsart}
 \usepackage{amsmath, amsthm, amscd, amsfonts, amssymb, graphicx, xcolor, 
 }
\usepackage{mathrsfs}
\usepackage{enumerate}

\usepackage[
  a4paper,
  margin=2.5cm,
  headsep=6mm
]{geometry}

\newtheorem{theorem}{Theorem}[section]
\newtheorem{lemma}[theorem]{Lemma}

\theoremstyle{definition}

\theoremstyle{remark}

\numberwithin{equation}{section}

\newcommand{\N}{\mathbb{N}} 

\newcommand{\C}{\mathbb{C}}

\usepackage[bookmarksnumbered, colorlinks]{hyperref}
\hypersetup{linkcolor=red, anchorcolor=green, citecolor=cyan, urlcolor=red, filecolor=magenta, pdftoolbar=true}

\begin{document}

\title[\tiny Frequently hypercyclic sequences of differential operators]{Frequently hypercyclic sequences of differential operators on the space of entire functions}


\author[Bernal]{L.~Bernal Gonz\'alez}
\address[L. Bernal-Gonz\'alez]{\mbox{}\newline \indent Departamento de An\'{a}lisis Matem\'{a}tico
\newline \indent Facultad de Matem\'{a}ticas
\newline \indent
	Instituto de Matem\'aticas Antonio de Castro Brzezicki
\newline \indent Universidad de Sevilla
\newline \indent Avenida Reina Mercedes s/n, 41012-Sevilla (Spain).}
\email{lbernal@us.es}


\author[Calder\'on]{M.C.~Calder\'on-Moreno}
\address[M.C.~Calder\'on-Moreno]{\mbox{}\newline \indent Departamento de An\'{a}lisis Matem\'{a}tico \newline
\indent Facultad de Matem\'{a}ticas \newline \indent
	Instituto de Matem\'aticas Antonio de Castro Brzezicki \newline
\indent Universidad de Sevilla \newline
\indent Avenida Reina Mercedes s/n, 41012-Sevilla (Spain).}
\email{mccm@us.es}


\author[Prado]{J.A.~Prado-Bassas}
\address[J.A.~Prado-Bassas]{\mbox{}\newline \indent Departamento de An\'{a}lisis Matem\'{a}tico \newline
\indent Facultad de Matem\'{a}ticas \newline \indent
	Instituto de Matem\'aticas Antonio de Castro Brzezicki \newline
\indent Universidad de Sevilla \newline
\indent Avenida Reina Mercedes s/n, 41012-Sevilla (Spain).}
\email{bassas@us.es}


\subjclass[2020]{30D15, 30K15, 47A16, 47B91}

\keywords{Differential operator, entire functions, hypercyclic sequence of operators, frequent hypercyclicity, Borel transform}


\begin{abstract}
A criterion to obtain frequent hypercyclicity for a sequence of convolution operators on the space of entire functions on the complex plane is provided.
The criterion involves that the generating functions of the operators do not vanish on an appropriate annulus, in the boundary of which the modulus of each term of the sequence is in some sense controlled by the preceding ones or the following ones.
\end{abstract}


\maketitle


\section{Introduction.}

Hypercyclicity is a subject in functional analysis that has been thoroughly investigated for the last four decades.
It deals with the search for vectors whose orbits under an operator or sequence of operators are dense in the supporting space.
For background on this topic the reader is referred to the survey \cite{Gr1} and the books \cite{bayartM,grosseP}.
An interesting class of operators to be studied under this point of view is the one of differential operators on the space of entire functions
on the complex plane $\C$. In this note, we shall consider a sequence of differential operators and study its frequent hypercyclicity --a stronger form of hypercyclicity, see below-- under a set of reasonable assumptions.

Let us fix the pertinent terminology. Let $X,Y$ be two linear topological spaces (in many cases, $X = Y$), $T_n :X \to Y$ ($n \in \N$) be a sequence of continuous linear mappings, and $x_0 \in X$. Then $x_0$ is said to be {\it hypercyclic} or {\it universal} for $(T_n)$ whenever its orbit $\{T_n x_0: \, n \in \N\}$ under $(T_n)$ is dense in $Y$. The family $(T_n)$ is called {\it hypercyclic} whenever it has a hypercyclic vector. It is
plain that, in order that a sequence $(T_n)$ can be hypercyclic, $Y$ must be separable. If $T:X \to X$ is an operator (= continuous linear selfmapping) on $X$, then a
vector $x \in X$ is said to be {\it hypercyclic} for $T$ provided that it is hypercyclic for the sequence $(T^n)$ of iterates of $T$, i.e., $T^n = T \circ T \circ \cdots \circ T$
($n$-fold). The operator $T$ is {\it hypercyclic} when there is a hypercyclic vector for $T$.

In this paper, we are specially interested in differential operators defined on the Fr\'echet space $H(\C )$ of all entire functions, endowed with the topology of uniform convergence in compacta. An entire function $\Phi (z) = \sum_{n=0}^\infty a_n z^n$ is said to be of exponential type whenever there
exist positive constants $A$ and $B$ such that $|\Phi (z)| \leq A e^{B|z|}$ $(z \in \C)$. The class of these functions will be denoted by ${\mathcal E}$.
It is easy to realize (see, for instance, \cite{Dic,BeG}) that if $\Phi$ is an entire function with exponential type as above, then the series
$\Phi (D) = \sum_{n=0} a_n D^n$ defines an operator on $H(\C)$, acting in the following way:
$$
(\Phi (D)f)(z) = \sum_{n=0} a_n f^{(n)}(z) \, \hbox{ \ for all \ } z \in \C .
$$
So $\Phi (D)$ defines, under the latter conditions, a (generally, infinite order) linear
differential operator with constant coefficients. It is shown
in \cite{GoS} that an operator $L$ on $H(\C)$
commutes with every translation operator $\tau_a$ ($\tau_a f := f(\cdot + a), \, a \in \C$) if and only if $L$ commutes with the
derivative operator $D$, and if and only if $L = \Phi (D)$ for some entire
function $\Phi \in \mathcal E$. Godefroy and Shapiro proved in \cite{GoS} the hypercyclicity of every non-scalar operator on $H(\C )$ commuting with translations.
These operators are also called convolution operators.
In particular, by choosing $\Phi (z) = e^{az}$ and $\Phi (z) = z$, one recovers the hypercyclicity of each nontrivial translation operator and of the derivative operator, that had been respectively obtained by Birkhoff \cite{Bir} and MacLane \cite{Mac}. As far as we know, the hypercyclicity of {\it sequences} of convolution operators, being not necessarily the iterates of a single one, was analyzed in \cite{Be3} for the first time.

A stronger form of hypercyclicity is that of frequent hypercyclicity. This notion was coined by Bayart and Grivaux \cite{bayartG} in 2006. An operator $T$ on a topological vector space $X$ is said to be {\it frequently hypercyclic} if there is a vector $x_0 \in X$, called frequently hypercyclic for $T$, such that
$\underline{\rm dens} \{n \in \N : \, T^n x_0 \in U\} > 0$
for any nonempty open subset $U$ of $X$, where $\underline{\rm dens} (A)$ represents the lower density of a subset $A \subset \N$, that is,
$$
\underline{\rm dens} (A) = \displaystyle{\liminf_{n \to \infty} {{\rm card} \, (A \cap \{1,2, \dots ,n\}) \over n}}.
$$
By replacing the iterates $T^n$ with the $n$th term $T_n$ of a sequence, we get the concept of a frequently hypercyclic sequence $(T_n)$ of operators, see \cite{BoniGr-b}.
As a stronger version of Godefroy--Shapiro's result, Bonilla and Grosse-Erdmann \cite{BoniGr-a} proved in 2006 that any non-scalar differential operator $\Phi (D)$ is, in fact, frequently hypercyclic.

The aim of this short note is to establish a criterion for the frequent hypercyclicity of sequences of convolution operators on the space of entire functions.
Such a criterion involves that the generating functions of the operators do not vanish on an appropriate annulus, in the boundary of which the modulus of each term of the sequence is in some sense controlled by the preceding ones or the following ones.
This will be carried out in Section 3, while Section 2 will be devoted to establish a number of auxiliary results on hypercyclicity, growth of entire functions, and Borel transform.

\section{Preliminary results.}

First of all, we recall a criterion for the frequent hypercyclicity of sequences of operators. The next sufficient condition
(Theorem \ref{Theorem Bonilla-Grosse}) was
established by Bonilla and Grosse-Erdmann, see \cite{BoniGr-b} and \cite[Chapter 9]{grosseP}. We need the notion of unconditional convergence. A series
$\sum_{n=1}^{\infty} x_n$ in a F-space $X$ with F-norm $\|\cdot\|$ is called {\it unconditionally convergent} if for any bijection
$\pi : \N \to \N$ the series $\sum_{n=1}^{\infty} x_{\pi (n)}$ converges. Conditions equivalent to unconditional convergence
can be found in \cite[3.8.2 and p.~153]{rolewicz} and \cite[3.3.8 and 3.3.9]{kamthangupta}. One of them, that will be used in the proof
of Theorem \ref{main theorem freq HC}, reads as follows: for any bounded sequence $(\alpha_n)$ of scalars, the series $\sum_{n=1}^{\infty} \alpha_n x_n$ converges.

Moreover, a family $\sum_{n=1}^\infty x_{n,k}$ ($k\in I\subset\N$) of series in $X$ is called unconditionally convergent {\em uniformly on $k$}, if given $\varepsilon>0$, there exists $N\in\N$ not depending on $k$, such that for every finite set $F\subset\N$ with $F\cap\{1,2,\cdots,N\}=\varnothing$ and every $k\in I$ it holds that $\left\|\sum_{n\in F}x_{n,k}\right\|<\varepsilon$.

\begin{theorem} \label{Theorem Bonilla-Grosse}
Let $X$ be an F-space, $Y$ a separable F-space
and $T_n : X \to Y$ a sequence of continuous linear mappings. Suppose that there are a dense subset $Y_0$ of $Y$
and mappings $S_n : Y_0 \to X$ such that, for all $y \in Y_0$, the following hold:
\begin{enumerate}[\rm (i)]
\item $\sum_{n=1}^{k} T_k S_{k-n} y$ converges unconditionally in $Y$, uniformly in $k \in \N$.
\item $\sum_{n=1}^{\infty} T_k S_{k+n} y$ converges unconditionally in $Y$, uniformly in $k \in \N$.
\item $\sum_{n=1}^{\infty} S_n y$ converges unconditionally in $Y$.
\item $T_n S_n y \to y$ as $n \to \infty$.
\end{enumerate}
Then the sequence $(T_n)$ is frequently hypercyclic.
\end{theorem}

In condition (i) we only have finite series, but one may understand them as
infinite series by adding zero terms. Moreover, it is understood that $S_0 =$ the identity on $Y_0$.

%

Now, we turn our attention to an important integral representation of entire functions.
Recall that the {\it growth order} of a function $\Phi
\in H(\C )$ is defined as \hfil\break
$\rho (\Phi ) := \limsup_{r \to \infty}
{\log \log M(r,\Phi )\over \log r}$, where $M(r,\Phi ) := \max
\{|f(z)|: \, |z| = r\}$ (see \cite{Boa}). If $\alpha \in
[0,+\infty )$, the {\it $\alpha$-type} of $\Phi$ is $\tau_\alpha
(\Phi ) := \limsup_{r \to \infty} {\log M(r,\Phi ) \over
r^\alpha}$. The {\it exponential type} of $\Phi$ is $\tau (\Phi )
:= \tau_1 (\Phi )$. Then $\Phi \in \mathcal{E}$ (that is, $\Phi$ is of exponential type) if and only
if $\tau (\Phi ) < +\infty$.

By the Malgrange-Ehrenpreis theorem (see \cite{BeG}, \cite{Ehr} or
\cite{Mal}), any nonzero differential operator $\Phi (D)$ is {\it surjective} and, in fact, it possesses a continuous right inverse.
In order to calculate such an inverse in the proof of our main theorem, we have at our disposal the tool of the Borel transform.

Assume that $\Phi \in \mathcal{E}$. The {\it Borel transform} of $\Phi$ is the function given
by the series
\begin{equation} \label{equation Borel transform}
(B \Phi )(z) = \sum_{n=0}^\infty {n! a_n \over z^{n+1}},
\end{equation}
provided that $\Phi (z) = \sum_{n=0}^\infty a_nz^n$. Since $\limsup_{n
\to \infty} |\Phi^{(n)}(0)|^{1/n} = \tau$ (see
\cite[pp.~11--12]{Boa}), the series in \eqref{equation Borel transform} defines an analytic
function in $\{z: \, |z| > \tau\}$. The next assertion,
usually refered to as P\'olya's representation, can be found in
\cite[pp.~73--74]{Boa}.

\begin{lemma} \label{Lemma representation Borel of Phi}
Suppose that $\Phi$ is an entire function of exponential type. Then,
for every $R > \tau (\Phi )$ and every $z \in \C$, we have
$$
\Phi (z) = {1 \over 2\pi i} \oint_{|t|=R} e^{zt} \, (B \Phi)(t) \, dt.
$$
\end{lemma}

%


\section{Frequent hypercyclicity of sequences of differential operators.}

Now, we provide a general criterion of frequent hypercyclicity for sequences of differential operators on $H(\C )$ generated by entire functions satisfying appropriate conditions.

\vskip 2pt

For any $R_1,R_2, \dots, R_p \in (0,+\infty )$ we shall denote by
$A(R_1, \dots ,R_p)$ the closed annulus centered at the origin determined by the radii $R_i$ ($i=1, \dots ,p$), that is,
\[A(R_1, \dots ,R_p) = \{z \in \C : \, \rho \le |z| \le \sigma\},\]
where $\rho := \min \{R_1, \dots ,R_p\}$ and $\sigma := \max \{R_1, \dots ,R_p \}$.

\begin{theorem} \label{main theorem freq HC}
Let $(\Phi_n)$ be a sequence of 
entire functions of exponential type.
Denote by $\Phi_0$ the constant function $1$.
Assume that there exist positive reals $R_1,R_2,R_3,$ \break
$\alpha_{n,j} \,(n > j \ge 0), \,  \beta_{n,j} \, (n > j \ge 1)$ satisfying the following conditions:
\begin{enumerate}[\rm (a)]
\item  $\displaystyle{\sup_{n \in \N} \sum_{j=0}^{n-1} \alpha_{n,j}} < +\infty$ and $\displaystyle{\sup_{j \in \N} \sum_{n=j+1}^{\infty} \beta_{n,j}} < +\infty$.
\item $\Phi_n(t) \ne 0$ for all $t \in A(R_1,R_2,R_3)$ and all $n \in \N$.
\item $|\Phi_n(t)| \le \alpha_{n,j} |\Phi_j(t)|$ for all $t \in \{|z|=R_1\}$ and all $n > j \ge 0$.
\item $|\Phi_j(t)| \le \beta_{n,j} |\Phi_n(t)|$ for all $t \in \{|z|=R_2\}$ and all $n > j \ge 1$.
\item The series $\displaystyle{\sum_{n=1}^{\infty} {1 \over \min_{|t| = R_3} |\Phi_n(t)| } }$ converges.
\end{enumerate}
Then the sequence of differential operators $(\Phi_n (D))$ is frequently hypercyclic on $H(\C )$.
\end{theorem}

\begin{proof}
Let us apply Theorem \ref{Theorem Bonilla-Grosse} with the following cast of characters:

\centerline{$X := H(\C ) =: Y$, $Y_0 := \{$polynomials$\}$, $T_n := \Phi_n (D)$ ($n \in \N$).}

\noindent Our goal is to define the mappings $S_n : Y_0 \to X$ so that
all four conditions (i) to (iv) in the mentioned theorem are fulfilled.

Since $\tau (P) = 0$ for every polynomial $P$, it follows from Lemma \ref{Lemma representation Borel of Phi} that for all $R > 0$ we have
\begin{equation} \label{Eq Borel polynomial}
P(z) = {1 \over 2\pi i} \oint_{|t|=R} e^{zt} \, (B P)(t) \, dt \hbox{ \ for all \ } z \in \C \hbox{ \ and all \ } P \in Y_0.
\end{equation}
Let $R \in [\rho , \sigma ]$, where $\rho = \min \{R_1,R_2,R_3\}$ and $\sigma = \max \{R_1,R_2,R_3\}$. Thanks to (b), we have $\Phi_n(t) \ne 0$ ($n \ge 1$)
whenever $|t| = R$. We then define $S_n : Y_0 \to X$ by
\begin{equation} \label{Eq def Sn}
(S_n P)(z) := {1 \over 2\pi i} \oint_{|t| = R} {e^{zt} \, (B P)(t) \over \Phi_n(t)} \, dt \hbox{ \ for all } z \in \C .
\end{equation}
Notice that, by well-known results on parametric integration, the last formula defines a genuine entire function.
Moreover, thanks to the Cauchy integral theorem and the assumption (b), the definition of $S_n$ in \eqref{Eq def Sn} does not depend on $R$.
In particular, we are allowed to take $R = R_1,R_2,R_3$.

Fix $n \in \N$ and let $\Phi_n(z) = \sum_{j=0}^{\infty} h_j z^j$. Let us use \eqref{Eq Borel polynomial} (with $R \in [\rho , \sigma ]$) and \eqref{Eq def Sn}.
Then for every $P \in Y_0$ and every $z \in \C$ we obtain that
\begin{equation*}
\begin{split}
(T_n S_n P)(z) &= \sum_{j=0}^\infty h_j (S_n P)^{(j)}(z) = \sum_{j=0}^\infty {h_j \over 2\pi i} \oint_{|t|=R} t^j e^{zt} {(B P)(t) \over \Phi_n(t)} \, dt \\
               &= {1 \over 2\pi i} \oint_{|t|=R} \left( \sum_{j=0}^\infty h_j t^j \right) e^{zt} {(BP)(t) \over \Phi_n(t)} \, dt \\
               &= {1 \over 2\pi i} \oint_{|t|=R} e^{zt} (BP)(t) \, dt = P(z).
\end{split}
\end{equation*}
Note that, in order to get the interchange of integration and summation at
the third equality, we have used $\sum_{j=0}^\infty \oint_{|t|=R}
|{h_j \over 2\pi i} t^j e^{zt} {(Bf)(t) \over \Phi_n(t)}| \, dt < +
\infty$, which follows easily from $|h_j|^{1/j} \to 0$, that in turn
holds because $\Phi_n$ is entire.

Summarizing, we have obtained for all $P \in Y_0$ that $T_n S_n P = P \to P$ ($n \to \infty$). This is (iv) in Theorem \ref{Theorem Bonilla-Grosse}.

Let us show condition (iii) in the mentioned theorem.  According to (e), we have
$$
\sum_{n=1}^{\infty} {1 \over \min_{|t| = R_3} |\Phi_n(t)| }  < +\infty .
$$
Fix a bounded sequence $(c_n)$ of scalars. It should be proved that, for each polynomial $P$, the series
$\sum_{n=1}^{\infty} c_n S_n P$ converges compactly in $\C$. In turn, plainly, it suffices to prove that, for every prescribed $K > 0$, we have
$\displaystyle{\sum_{n=1}^{\infty} \sup_{|z| \le K} |(S_n P) (z)|} < +\infty$. With this aim, observe that (e) implies
\begin{equation*}
\begin{split}
\sum_{n=1}^{\infty} \sup_{|z| \le K} |(S_n P) (z)| &= \sum_{n=1}^{\infty} \sup_{|z| \le K} \left| {1 \over 2\pi i} \oint_{|t|=R_3} {e^{zt} \, (B P)(t) \over \Phi_n(t)} \, dt
                                                                                                                                                                 \right| \\
&\le {1 \over 2\pi} \cdot 2 \pi R_3 \cdot \sum_{n=1}^{\infty} \sup_{|z| \le K} \sup_{|t| = R_3} \left| {e^{zt} \, (B P)(t) \over \Phi_n(t)} \right| \\                                                &\le R_3 \cdot e^{K R_3} \cdot \sup_{|t|=R_3} |(B P)(t)| \cdot \sum_{n=1}^{\infty} {1 \over \min_{|t| = R_3} |\Phi_n(t)| }  < +\infty ,
\end{split}
\end{equation*}
from which the desired result comes.

Now, we pass to prove condition (ii) in Theorem \ref{Theorem Bonilla-Grosse}. Fix a polynomial $P$ and a bounded sequence $(c_n)$ of scalars.
We want to show that the series $\sum_{n=1}^{\infty} c_n T_k S_{n+k} P$ converges compactly in $\C$ uniformly in $k \in \N$.
For this, it is enough to prove that, for every $K > 0$, the series
$\displaystyle{\sum_{n=1}^{\infty} \sup_{k \in \N} \sup_{|z| \le K} |(T_k S_{n+k}) (z)|}$ converges.
A similar argument to that used above to reach (iv) leads us to
$$
(T_k S_{n+k} P)(z) = {1 \over 2\pi i} \oint_{|t|=R_2} e^{zt} (BP)(t) \cdot {\Phi_k (t) \over \Phi_{n+k} (t)} \, dt \hbox{ \ for all } z \in \C.
$$
Therefore,
\begin{equation*}
\begin{split}
\sum_{n=1}^{\infty} \sup_{|z| \le K \atop k \in \N} |(T_k S_{n+k} P) (z)| &= \sum_{n=1}^{\infty} \sup_{|z| \le K \atop k \in \N}
\left| {1 \over 2\pi i} \oint_{|t|=R_2} e^{zt} (BP)(t) \cdot {\Phi_k (t) \over \Phi_{n+k} (t)} \, dt \right| \\
&\le {1 \over 2\pi} \cdot 2 \pi R_2 \cdot \sum_{n=1}^{\infty} \sup_{|z| \le K, \, |t| = R_2 \atop k \in \N} \left| {e^{zt} \, (B P)(t) \Phi_k(t) \over \Phi_{n+k}(t)} \right| \\                                                &\le R_2 \cdot e^{K R_2} \cdot \sup_{|t|=R_2} |(B P)(t)| \cdot \sup_{k \in \N} \sum_{n=1}^{\infty} \beta_{n+k,k} \\
&= R_2 \cdot e^{K R_2} \cdot \sup_{|t|=R_2} |(B P)(t)| \cdot \sup_{k \in \N} \sum_{n=k+1}^{\infty} \beta_{n,k} < +\infty,
\end{split}
\end{equation*}
where we have used (d) in the second inequality, and then (a) to conclude the finiteness of the last series.
So, condition (ii) is proved.

Finally, we are going to show that condition (i) in Theorem \ref{Theorem Bonilla-Grosse} holds for our sequence $(T_n)$.
Fix, again, a polynomial $P$ and a bounded sequence $(c_n)$ of scalars.
Our goal is to prove that $\sum_{n=1}^{k} c_n T_k S_{k-n} P$ converges compactly in $\C$ uniformly in $k \in \N$.
It suffices to show that, for every $K > 0$, the extended positive number 
$\displaystyle{\sup_{k \in \N} \sum_{n=1}^{k} \sup_{|z| \le K} |(T_k S_{k-n}) (z)|}$ is finite.
Let us remark that the convention $S_0 =$ the identity is consistent with \eqref{Eq def Sn} if one takes into account \eqref{Eq Borel polynomial} and the fact $\Phi_0 \equiv 1$.
As in the preceding parts, we obtain the following estimations, where (c) and the first part of (a) are implemented:
\begin{equation*}
\begin{split}
\sup_{k \in \N} \sum_{n=1}^{k} \sup_{|z| \le K} |(T_k S_{k-n}) (z)| &=
\sup_{k \in \N} \sum_{n=1}^{k} \sup_{|z| \le K} \left| {1 \over 2\pi i} \oint_{|t|=R_1} e^{zt} (BP)(t) \cdot {\Phi_k (t) \over \Phi_{k-n} (t)} \, dt \right| \\
&\le {1 \over 2\pi} \cdot 2 \pi R_1 \cdot \sup_{k \in \N} \sum_{n=1}^{k} \sup_{|z| \le K \atop |t| = R_1} \left| {e^{zt} \, (B P)(t) \Phi_k(t) \over \Phi_{k-n}(t)} \right| \\                                                &\le R_1 \cdot e^{K R_1} \cdot \sup_{|t|=R_1} |(B P)(t)| \cdot \sup_{k \in \N} \sum_{n=1}^{k} \beta_{k,k-n} \\
&= R_1 \cdot e^{K R_1} \cdot \sup_{|t|=R_1} |(B P)(t)| \cdot \sup_{k \in \N} \sum_{n=0}^{k-1} \beta_{k,n} < +\infty .
\end{split}
\end{equation*}
Hence, condition (i) is proved and the proof of our theorem is concluded.
\end{proof}

It is fair to say that, at first glance, the conditions (a) to (e) assumed in the previous theorem might have been
conceived {\it ad hoc} in order to apply Theorem \ref{Theorem Bonilla-Grosse}.
Nevertheless, we are going to see in the following result how Theorem \ref{main theorem freq HC} can be used to give
concrete examples of frequently hypercyclic sequences of
convolution operators whose members are not the iterates of a single operator.


\begin{theorem} \label{freq HC special cases}
Let $(c_n)$ be a sequence of nonzero complex scalars satisfying
\begin{equation} \label{Eq liminf limsup}
\gamma := \liminf_{n \to \infty} \left| {c_{n+1} \over c_n}  \right| > 0 \hbox{ \ and \ } \delta := \limsup_{n \to \infty} \left| {c_{n+1} \over c_n}  \right| < +\infty .
\end{equation}
Then the following holds:
\begin{enumerate}[\rm (i)]
\item Assume that $\Phi \in \mathcal{E}$ is a function for which there exist $R_1,R_2 \in (0,+\infty)$ such that $|\Phi (t)| < 1/\delta$ on $|t| = R_1$,
               $|\Phi (t)| > 1/\gamma$ on $|t| = R_2$, and $\Phi (t) \ne 0$ for all $t \in A(R_1,R_2)$. Then the sequence of differential operators
               $(c_n \Phi(D)^n)$ is frequently hypercyclic on $H(\C )$.
\item The sequence of differential operators $(c_n D^n)$ is frequently hypercyclic on $H(\C )$.
\end{enumerate}
\end{theorem}

\begin{proof}
Observe that (ii) is derived from (i) just taking $\Phi (z) := z$, $R_1 = {1 \over 2\delta}$ and $R_2 := {2 \over \gamma}$.
Therefore, it is enough to prove (i).

With this aim, consider the functions $\Phi_n(z) := c_n \Phi^n (z)$ $(n = 1,2, \dots )$, so that $\Phi_n(D) = c_n \Phi (D)^n$.
It follows from \eqref{Eq liminf limsup}, the continuity of $|\Phi |$ and the compactness of $|t| = R_i$ $(i=1,2)$
that there exist $\mu, \nu > 0$ such that $\mu > \delta$, $\nu < \gamma$, $|\Phi (t)| \le {1 \over \mu}$ on $|t| = R_1$,
and $|\Phi (t)| \ge {1 \over \nu}$ on $|t| = R_2$. Fix $\mu' \in (\delta , \mu)$ and $\nu' \in (\nu ,\gamma )$.
Then there is
$n_0 \in \N$ such that
$\nu' < \left|{c_{n+1} \over c_n} \right| < \mu'$ for all $n \ge n_0$. By deleting, if necessary, finitely many terms from our sequence of operators, we can
assume without loss of generality that $n_0 = 1$. We set $c_0 := 1$ and $\Phi_0 (z) := 1$.
Then we have for all $n > j \ge 0$ and all $t \in \{|z|= R_1\}$ that
$$
|\Phi_n(t)| = |c_n \Phi^n(t)| = \left| {c_n \over c_{n-1}} \cdot {c_{n-1} \over c_{n-2}} \cdots {c_{j+1} \over c_j} \cdot \Phi^{n-j}(t) \right| \cdot |c_j \cdot \Phi^j(t)|
\le \left( {\mu' \over \mu} \right)^{n-j} |\Phi_j(t)|.
$$
Analogously, we obtain for all $n > j \ge 0$ and all $t \in \{|z|= R_2\}$ that
$$
|\Phi_n(t)| = |c_n \Phi^n(t)| = \left| {c_n \over c_{n-1}} \cdot {c_{n-1} \over c_{n-2}} \cdots {c_{j+1} \over c_j} \cdot \Phi^{n-j}(t) \right| \cdot |c_j \cdot \Phi^j(t)|
            \ge \left( {\nu' \over \nu} \right)^{n-j} |\Phi_j(t)|.
$$

In particular, $\min_{|t|=R_3} |\Phi_n (t)| \ge \left( {\nu' \over \nu} \right)^{n}$, where $R_3 := R_2$. Then $\Phi_n(t) \ne 0$ for all $n \in \N$ and all
$t \in A(R_1,R_2,R_3) = A(R_1,R_2)$, which is (b) in Theorem \ref{main theorem freq HC}. Moreover, (e) holds because, since $\nu' > \nu$, the series
$\sum_{n=1}^{\infty} {1 \over \min_{|t|=R_3} |\Phi_n (t)|}$ converges.
Now, letting
$$
\alpha_{n,j} := \left( {\mu' \over \mu} \right)^{n-j} \hbox{ \ and \ } \beta_{n,j} := \left( {\nu \over \nu'} \right)^{n-j},
$$
we get (c) and (d) satisfied in Theorem \ref{main theorem freq HC}.

Finally, we will check the remaining condition (a). Taking into account that ${\mu' \over \mu} < 1$ and  ${\nu \over \nu'} < 1$, we get:
\begin{itemize}
\item $\displaystyle{\sup_{n \in \N}} \sum_{j=0}^{n-1} \alpha_{n,j} = \displaystyle{\sup_{n \in \N}} \sum_{j=0}^{n-1} \left( {\mu' \over \mu} \right)^{n-j}
= \displaystyle{\sup_{n \in \N}} \sum_{j=1}^{n} \left( {\mu' \over \mu} \right)^{j} = \sum_{j=1}^{\infty} \left( {\mu' \over \mu} \right)^{j} = {\mu' \over \mu - \mu'} < +\infty .$
\item $\displaystyle{\sup_{j \in \N}} \sum_{n=j+1}^{\infty} \beta_{n,j} = \displaystyle{\sup_{j \in \N}} \sum_{n=j+1}^{\infty} \left( {\nu \over \nu'} \right)^{n-j}
= \sum_{n=1}^{\infty} \left( {\nu \over \nu'} \right)^{n} = {\nu \over \nu' - \nu} < +\infty .$
\end{itemize}
The proof is finished.
\end{proof}

Part (ii) of the preceding theorem furnishes a collection of frequently hypercyclic sequences of finite order differential operators. Let us provide an example made with infinite order ones. Consider $T_n := \log n \cdot \left(D + {\mathcal{T} \over 3}\right)^n$ $(n \in \N)$, where $\mathcal{T}$ is the $1$-translation operator $(\mathcal{T}f)(z) := f(z+1)$.
Then $D + {\mathcal{T} \over 3} = \Phi (D)$ with $\Phi (z) := z + {e^z \over 9}$. Moreover, if $c_n := \log n$, then under the notation of \eqref{Eq liminf limsup} we have $0 < 1 = \gamma = \delta < +\infty$. Let us apply Rouch\'e's theorem (see, e.g., \cite[p.~125]{conway}) on the circles $\gamma_1 \equiv |z| = 1/2$
and $\gamma_2 \equiv |z| = 2$. We have that the origin is the only zero of $z$ and $|e^z/9| < 2/9 < 1/2 = |z|$ on $\gamma_1$, while in $\gamma_2$ it holds that
$|e^z/9| < 8/9 < 2 = |z|$ on $\gamma_2$. Then $\Phi$ has a unique zero in $|z| \le 1$ and $\Phi (z) \ne 0$ for all $z \in A(1,2)$. In addition,
$|\Phi (t)| \le |z| + |e^z/9| < 1/2 + 2/9 < 1 = 1/\delta$ on $\gamma_1$, and $|\Phi (t)| \ge |z| - |e^z/9| > 2 - 8/9 > 1 = 1/\gamma$ on $\gamma_2$.
Thus, Theorem \ref{freq HC special cases}(i) applies, so yielding the frequent hypercyclicity of $(T_n)$.

We also provide an example of a frequently hypercyclic sequence $(T_n)$ of convolution operators being not of the form $(c_n \Phi (D)^n)$.
If $\mathcal{T}$ denotes, again, the $1$-translation operator, we define $T_n := 5^n D^n + 9^{-n}\mathcal{T}^n$ $(n \in \N )$.
Then $T_n = \Phi_n(D)$, where $\Phi_n(z) = 5^nz^n + 9^{-n}e^{nz}$.
Now, apply Rouch\'e's theorem on the closed arcs $\gamma_1 \equiv |z| = 1/15$
and $\gamma_2 \equiv |z| = 1$ to conclude that $\Phi_n(t) \ne 0$ for all $t \in A(1/15,1)$. Let $\Phi_0(z) := 1$.
Simple calculations ---\,the details are left to the interested reader\,--- lead us to $|\Phi_n(t)| \le 2^{j-n} |\Phi_j(t)|$
for all $t \in \{|z|=1/15\}$ and all $n > j \ge 0$, and
$|\Phi_j(t)| \le 2^{j-n} |\Phi_n(t)|$ for all $t \in \{|z|=1\}$ and all $n > j \ge 0$. Consequently, conditions (a) to (e)
in Theorem \ref{main theorem freq HC} are fulfilled by selecting $R_1 = 1/15$, $R_2 = 1 = R_3$, and so $(T_n)$ is frequently hypercyclic.

\section{Final remarks and questions.}

\noindent 1. Assume that $\Phi \in \mathcal{E}$ satisfies the assumptions in Theorem \ref{freq HC special cases}(i).
Taking into account that $\liminf_{n \to \infty} \left|{ c_{n+1}  \over c_n} \right| \le \liminf_{n \to \infty} |c_n|^{1/n} \le \limsup_{n \to \infty} |c_n|^{1/n} \le \liminf_{n \to \infty} \left|{ c_{n+1}  \over c_n} \right|$,
the following question arises naturally:

\begin{quote}
  {\it Does $0 < \liminf_{n \to \infty} |c_n|^{1/n} \le \limsup_{n \to \infty} |c_n|^{1/n} < +\infty$ suffice to\\ guarantee frequent hypercyclicity for $(c_n\Phi (D))$?}
\end{quote}

\noindent This should be compared with Theorem 2.6 in \cite{bernalprado} containing the next assertion.
Assume that at least one of the following properties is fulfilled:
\begin{itemize}
\item $(|c_n|^{1/n})$ does not converge to zero and $\Phi (0) = 0$.
\item $0 < \liminf_{n \to \infty} |c_n|^{1/n} \leq \limsup_{n \to \infty} |c_n|^{1/n} < +\infty$.
\end{itemize}

\noindent Then the sequence $(c_n \Phi(D)^n)$ is hypercyclic on $H(\C )$.

\vskip 6pt

\noindent 2. In \cite[Theorem 4]{Be2} it is shown that $(c_nD^n)$ is hypercyclic on $H(\C )$ if and only if the sequence $(n|c_n|^{1/n})$ is unbounded.
This is evidently not enough for frequent hypercylicity, because we might have that $\{n|c_n|^{1/n} : \, n \in S\}$ is bounded, $S$ being a subset of $\N$
such that $\underline{\rm dens} (\N \setminus S) = 0$. By reinforcing the hypothesis, we pose the following problem:

\begin{quote} {\it If $\lim_{n \to \infty} n|c_n|^{1/n} = \infty$, is $(c_nD^n)$ frequently hypercyclic?}
\end{quote}

\noindent 3. Several hypercyclicity criteria for sequences of operators of the form $(c_n(\cdot ) D^n)$ can be found in \cite{Cal}, where this time
the $c_n$'s are entire functions. A corresponding study of frequent hypercyclicity could be interesting.

\end{document}